\documentclass[runningheads]{llncs}

\usepackage{amsmath}
\usepackage{amssymb}
\usepackage[T1]{fontenc}
\usepackage[utf8]{inputenc}
\usepackage{xcolor}
\usepackage{hyperref}
\hypersetup{
    colorlinks=true,
    linkcolor=blue,
    urlcolor=blue,
    citecolor=blue
}
\usepackage{verbatim}

\newcommand{\og}{\overrightarrow{G}}     
\newcommand{\ex}{\operatorname{ex}}

\begin{document}

\title{An Improved Upper Bound for the Tur\'an Number of the Hexagon}
\titlerunning{An Improved Upper Bound for $\ex(n,C_6)$}

\author{
Sandip Das\inst{1} \and
Sk Samim Islam\inst{1} \and
Aashirwad Mohapatra\inst{1} \and
Saumya Sen\inst{1}
}

\authorrunning{S. Das, S.S. Islam, A. Mohapatra, and S. Sen}

\institute{
Indian Statistical Institute, Kolkata, India\\
\email{
sandip.das.69@gmail.com,
samimislam08@gmail.com,
aashirwad\_r@isical.ac.in,
saumyasen72@gmail.com
}
}

\maketitle

\begin{abstract}
For a graph $F$, the Tur\'an number $\ex(n,F)$ is the maximum number of
edges in an $n$-vertex graph containing no  copy of $F$. Determining the
Tur\'an numbers of even cycles is a central problem in extremal graph theory
and remains open in general. For $C_6$, the best previous upper bound was
due to F\"uredi, Naor, and Verstra\"ete [\textit{Advances
in Mathematics, 2006}], who proved that, for sufficiently large positive integer $n$, $$
    \ex(n,C_6) \leq \lambda n^{4/3}+O(n)<0.6272 n^{4/3},
$$ where $\lambda$ is the real root of $
    16\lambda^3-4\lambda^2+\lambda-3=0$. We improve this bound by showing that, for sufficiently large positive integer $n$, $$
    \ex(n,C_6) \leq \alpha n^{4/3}+O(n)<0.6144 n^{4/3},
$$ where $\alpha$ is the unique real root of $
    4 \alpha^{3}
(3/2)^{1-1/(2\alpha)} =1$ in the interval $(1/2,2/3)$.

\end{abstract}

\section{Introduction}

For a graph $F$, the Tur\'an number $\ex(n,F)$ is the maximum number of
edges in an $n$-vertex graph containing no isomorphic  copy of $F$. This
function is one of the foundational topics of extremal graph theory. Tur\'an's
classical theorem ~\cite{Turan_classical} determines $\ex(n,K_r)$ exactly for every
complete graph $K_r$. More generally, Tur\'an-type problems for non-bipartite
forbidden graphs are comparatively well understood. This is particularly
striking for odd cycles.  Füredi and Gunderson~\cite{FZG2015},
building on earlier works of Bondy \cite{bondy1971large,bondy1971pancyclic}, Woodall \cite{woodall1972sufficient}, and Bollob\'as \cite{bollobas2004extremal}, described the extremal $C_{2k+1}$-free graphs for all $n$. In particular, for $n\ge 4k$,
$$
\ex(n,C_{2k+1})=\left\lfloor\frac{n^2}{4}\right\rfloor,
$$
and equality is attained only by $K_{\lfloor n/2\rfloor,\lceil n/2\rceil}$  up to isomorphism, i.e., the balanced complete bipartite graph is the unique extremal graph.

For even cycles, the corresponding extremal problems are substantially more difficult. Although $C_{2k}$ is one
of the simplest bipartite graphs, the exact value of $\ex(n,C_{2k})$ is not
known in general even for $C_4$ or $C_6$. The case of $C_4$ is nevertheless
understood asymptotically. Erd\H{o}s and R\'enyi~\cite{erdos_renyi_1962}, and
independently Brown~\cite{Brown1966}, constructed $C_4$-free graphs giving the
asymptotically sharp lower bound, while Erd\H{o}s, R\'enyi, and
S\'os~\cite{ERS1966} proved the corresponding upper bound. Together,
their results give
\[
    \ex(n,C_4) = \left(\frac{1}{2} + o(1) \right) n^{3/2}.
\]

For longer even cycles,  the  problem becomes substantially
more difficult. Bondy and Simonovits~\cite{BondySimonovits1974} proved that
$$
\ex(n,C_{2k})=O\bigl(n^{1+1/k}\bigr)
$$
for every fixed $k\ge 2$, establishing the exponent $1+1/k$ on the
upper-bound side. Despite these general upper bounds, matching lower bounds of order
$n^{1+1/k}$ are known only for $k\in\{2,3,5\}$, through the constructions
based on finite geometries~\cite{Brown1966,LUW1994,FNV2006}; see also
the survey~\cite{FurediSimonovits2013}.

The asymptotic formula for $C_4$ led Erd\H{o}s and
Simonovits~\cite{ErdosSimonovits1982} to conjecture that the same leading
constant should persist for every fixed even cycle:
\begin{equation}\label{eq:ES}
\ex(n,C_{2k})
=\left(\frac12+o(1)\right)n^{1+1/k},
\qquad k\ge 2.
\end{equation}

The conjecture turned out to be false. Lazebnik, Ustimenko, and
Woldar~\cite{LUW1994} first disproved it for $k=5$ by constructing
$C_{10}$-free graphs satisfying
$$
\ex(n,C_{10})
\ge
\left(4\cdot 5^{-6/5}+o(1)\right)n^{6/5}=
\left(0.5798\ldots+o(1)\right)n^{6/5}
$$
along an infinite sequence of values of $n$. Thus the constant $1/2$ in
\eqref{eq:ES} is already not correct for $C_{10}$.

The next important case was the hexagon.  Füredi, Naor, and
Verstra\"ete~\cite{FNV2006} disproved the $k=3$ instance of
\eqref{eq:ES}. They constructed, for infinitely many $n$, $C_6$-free graphs with more than $0.5338n^{4/3}$ edges. Thus, $\ex(n,C_6)>0.5338n^{4/3}$ for infinitely many $n$, which is the best known lower bound. In the same paper they proved the upper bound
\begin{equation}\label{eq:FNVbound}
\ex(n,C_6)
\le
\lambda n^{4/3} + O(n),
\end{equation}
where $\lambda$ is the real root of $16\lambda^3-4\lambda^2+\lambda-3=0$. Since $\lambda=0.627110\ldots$, this gives $\ex(n,C_6)<0.6272n^{4/3}$ for all sufficiently large $n$. This has remained the best known upper bound for the Tur\'an number of $C_6$.

In this paper, we improve the above mentioned upper bound for $C_6$ by lowering the leading coefficient.

\begin{theorem}\label{thm:main} For sufficiently large positive integer $n$, $$
    \ex(n,C_6) \leq \alpha n^{4/3}+O(n)<0.6144 n^{4/3},
$$ where $\alpha$ is the unique solution of $
    4 \alpha^{3}
(3/2)^{1-1/(2\alpha)} = 1$ in the interval $(1/2,2/3)$.

\end{theorem}

\section{Preliminaries}
\label{sec:preliminaries}

We begin with some notation and terminology that will be used throughout the paper.

Let $G=(V,E)$ be a simple graph and let $\pi=\{x,y\}\in\binom{V}{2}$. We write $C(\pi)$ for the number of paths of length three with end-vertices $x$ and $y$. Suppose that $C(\pi)=k\geq 2$, and denote these paths by $P_1,P_2,\ldots,P_k$. Let $\Pi$ consist of all pairs $\pi\in\binom{V}{2}$ such that either $C(\pi)\geq 2$ and $\pi\in E$, or $C(\pi)\geq 3$. We call $\pi$ \emph{degenerate} if the distance between $x$ and $y$ in $P_1\cup P_2\cup\cdots\cup P_k$ is two. Otherwise, $\pi$ is called \emph{non-degenerate}. Let $\Pi^{1}$ consist of all degenerate pairs $\pi\in\binom{V}{2}$.

Consider an orientation of $G$. A path $P$ of length three between $x$ and $y$ is called a \emph{forward path} if $P$ contains a directed path of length two starting at either $x$ or $y$. For a given orientation of $G$, let $\Pi^{2}$ consist of all pairs $\pi$ such that $C(\pi) \geq 2$ and at least one of the $C(\pi)$ length-three paths joining $\pi$ is a forward path.

To prove our result, we need some bounds on the degrees of a $C_6$-free graph and on the number of paths of length three joining pairs of vertices. The following results of F\"uredi, Naor, and Verstra\"ete \cite{FNV2006} give the bounds that we need. The first result relates the minimum and maximum degrees, while the second bounds the total contribution of the pairs in $\Pi \cup \Pi^{1}$. The third result shows the existence of an orientation that bounds the possible three length paths between pair of vertices in $\Pi^2$. We collect these results in the following lemma.

\begin{lemma}[\cite{FNV2006}]
\label{lem:FNV_collection}
Let $G = (V,E)$ be a $C_6$-free graph of minimum degree $\delta$ and maximum degree $\Delta$. Then the following hold:
\begin{itemize}
    \item[$(a)$] $\Delta (\delta - 4)^2 \leq 64 |V|$.
    
    \item[$(b)$] $\sum_{\pi \in \Pi \cup \Pi^{1}} C(\pi) \leq 35 \Delta |E|$.

    \item[$(c)$] There is an orientation of $G$ such that $\sum_{\pi \in \Pi^{2}} C(\pi) < 43 \Delta |E|$.  
\end{itemize}

\end{lemma}

The corresponding bipartite extremal problem is closely related to the hexagon problem. Such bipartite
estimates are particularly relevant to the ordinary Tur\'an problem, since a maximum cut extracts a large bipartite subgraph from an arbitrary graph. Let $\ex(n_1,n_2,C_6)$ denote the maximum number of edges in a $C_6$-free bipartite graph whose two partite sets have sizes $n_1$ and $n_2$. The following result of F\"uredi, Naor, and Verstra\"ete \cite{FNV2006} gives the upper bound that we need.

\begin{lemma}[\cite{FNV2006}]
\label{lem:FNV_thm1.2}
Let $n_1,n_2 \in \mathbb{N}$. Then $$\ex(n_1,n_2,C_6) < 2^{1/3} (n_1n_2)^{2/3} + 16 (n_1+n_2).$$
\end{lemma} 

\section{Proof of Theorem \ref{thm:main}}
\label{sec:proof}

Let $G = (V,E)$ be a simple graph. Let $n = |V|$ and $m = |E|$. Let $\delta$ and $\Delta$ denote the minimum and maximum degrees of $G$, respectively.
Let $b(G)$ be the maximum number of edges in a cut of $G$. Let $p_3(G)$ be the number of simple paths of length three, and let $u(G)$ be the number of unordered vertex pairs joined by exactly one such path.  

Let $\og$ be an orientation of $G$ and let $f(\og)$ denote the number of forward paths. For every vertex $v\in V$, let $N(v)=\{u\in V:uv\in E\}$ denote the neighborhood of $v$ in $G$, and let $d_v=|N(v)|$ denote its degree. Let $d^{-}_v$ and $d^{+}_v$ denote the in-degree and out-degree of $v$ in $\og$, respectively. Here $d_v=d^{-}_v+d^{+}_v$ 
and define
$$
    s(\og)=\sum_{u \to v }(d_u+d^{-}_u-1)(d_v-1),
$$
where the sum ranges over the directed edges of $\og$.

\begin{lemma} \label{lem:consequences_of_FNV}
If $G$ is $C_6$-free, then the following statements hold:

\begin{itemize}

\item[(a)] $b(G)\leq \frac{1}{2} n^{4/3} + 16 n$.

\item[(b)] $p_3(G) + u(G) \leq n(n-1) + 35 \Delta m$.  

\item[(c)] There is an orientation $\og$ of $G$ such that $f(\og)\leq u(G)+43\Delta m$. Consequently, $p_3(G) + f(\og) \leq n(n-1) + 78 \Delta m$ and $s(\og)\leq n(n-1)+80\Delta m$.
\end{itemize}

\end{lemma}

\begin{proof}
For a cut whose two parts have sizes $k$ and $n-k$ respectively, Lemma \ref{lem:FNV_thm1.2}
implies that the number of edges in that cut is less than $2^{1/3} (k(n-k))^{2/3} + 16 n$ which is at most $\frac{1}{2} n^{4/3} + 16n$.

Let $\mathcal{P}_i$ denote the set of unordered vertex pairs joined by exactly $i$ paths of length three. Every pair in $\mathcal{P}_i$ with $i\geq 3$ belongs to $\Pi$. Therefore, Lemma~\ref{lem:FNV_collection}(b) gives
$
\sum_{i\geq 3} i|\mathcal{P}_i|
\leq
\sum_{\pi\in\Pi\cup\Pi^{1}}C(\pi)
\leq 35\Delta m.
$
Consequently, $p_3(G) = |\mathcal{P}_1|+2|\mathcal{P}_2|+\sum_{i\geq3}i|\mathcal{P}_i| \leq |\mathcal{P}_1|+2|\mathcal{P}_2|+35 \Delta m \leq n(n-1) - u(G) + 35\Delta m$, where the last inequality follows since $|\mathcal{P}_1|+|\mathcal{P}_2|\leq \binom{n}{2}$ and $|\mathcal{P}_1| = u(G)$.

By Lemma~\ref{lem:FNV_collection}(c),  there is an orientation $\og$ of $G$ such that $\sum_{\pi \in \Pi^{2}} C(\pi) < 43 \Delta |E|$. The forward paths whose end-vertices are joined by exactly one path contribute at most $u(G)$. Every remaining forward path has an endpoint pair in $\Pi^{2}$. Hence
$
f(\og)
\leq
u(G)+\sum_{\pi\in\Pi^{2}}C(\pi)
<u(G)+43\Delta m
$.
Combining this with Lemma \ref{lem:consequences_of_FNV}(b) yields $p_3(G)+f(\og)
\leq n(n-1)+78\Delta m$.

Let $T$ be the number of triangles of $G$. For $uv\in E$, let $c_{uv} = |N(u) \cap N(v)|$. The number of simple paths whose middle edge is $uv$ equals
$(d_u-1)(d_v-1)-c_{uv}$. Since every triangle contributes one common neighbour
to each of its three edges,
\begin{equation}\label{eq:p3count}
  p_3(G)=\sum_{uv \in E}(d_u-1)(d_v-1)-3T .
\end{equation}

For a directed edge $u\to v$, a forward path having $uv$ as its middle edge must be forward from the endpoint adjacent to $u$. Hence it has the form $x\to u\to v-y$, where the orientation of the edge $vy$ may be arbitrary. There are initially $d^{-}_u(d_v-1)$ choices. A choice fails to give a simple path
precisely when $x=y$. Let $c^{-}_{uv}$ be the number of vertices
$x\in N(u)\cap N(v)$ for which $x\to u$. Then
\begin{equation}\label{eq:Fcount}
  f(\og) = \sum_{u \to v} d^{-}_u(d_v  - 1) -\sum_{u \to v} c^{-}_{uv}.
\end{equation}
A path $x-u-v-y$ can be forward from the $x$-endpoint only when the
middle edge is oriented $u\to v$, and it can be forward from the $y$-endpoint
only when the middle edge has the opposite orientation. Thus a forward path is
counted exactly once in \eqref{eq:Fcount}.
Each triangle contributes to at most three terms in $\sum_{u\to v}c^{-}_{uv}$,
and hence $\sum_{u\to v}c^{-}_{uv} \leq 3T$.
Adding \eqref{eq:p3count} and \eqref{eq:Fcount} gives $p_3(G) + f(\og) \geq s(\og) - 6T.$
Moreover, $3T = \sum_{uv \in E}c_{uv} \leq (\Delta - 1)m$ and thus $6T \leq 2 \Delta m$. The result then follows since $p_3(G)+f(\og) \leq n(n-1)+78\Delta m$. 
\qed
\end{proof}

Lemma~\ref{lem:consequences_of_FNV} gives an upper bound for $s(\og)$ when $G$ is $C_6$-free. We now give a lower bound for $s(\og)$ that holds for every
graph of minimum degree at least two. This bound is expressed in terms
of the number of edges and the size of a maximum cut. Comparing these
upper and lower bounds will give the main inequality used later.

\begin{lemma}\label{lem:oriented_lb}
If $G$ is a graph of minimum degree at least two, then every orientation $\og$ of $G$
satisfies
$$
    s(\og)
    \geq
    m\left(\frac{2m}{n}-1\right)^2
    \left(\frac{3}{2}\right)^{1-b(G)/m}.
$$
\end{lemma}

\begin{proof}
Since $s(\og)=\sum_{u\to v}(d_u+d^{-}_u-1)(d_v-1)$ and all the summands are positive, the AM--GM inequality gives
\begin{align}
    \log\frac{s(\og)}{m}
    &\geq
    \frac{1}{m}\sum_{u\to v}
    \log\bigl[(d_u+d^{-}_u-1)(d_v-1)\bigr]
    \notag\\
    &=
    \frac{1}{m}\sum_{v\in V}
    \bigl[d^{+}_v\log(d_v+d^{-}_v-1)+d^{-}_v\log(d_v-1)\bigr].
    \label{eq:AMGM}
\end{align}

Since $d_v\geq 2$, we have
\begin{align*}
    &d^{+}_v\log(d_v+d^{-}_v-1)+d^{-}_v\log(d_v-1)\\
    &\qquad=
    d_v\log(d_v-1)
    +d^{+}_v\log\left(1+\frac{d^{-}_v}{d_v-1}\right)\\
    &\qquad\geq
    d_v\log(d_v-1)
    +d_v\left(1-\frac{d^{-}_v}{d_v}\right)
    \log\left(1+\frac{d^{-}_v}{d_v}\right)\\
    &\qquad\geq
    d_v\log(d_v-1)
    +\log\left(\frac{3}{2}\right)\min\{d^{-}_v,d^{+}_v\},
\end{align*}
where the last inequality is because
$
    (1-t)\log(1+t)
    \geq
    \log\left(\frac{3}{2}\right)\min\{t,1-t\},
$ for every $t\in[0,1]$. For $t\in[1/2,1]$, this immediately follows since $\log(1+t)\geq\log(3/2)$. For $t\in[0,1/2]$, the difference
$
    (1-t)\log(1+t)-t\log(3/2)
$
is concave and vanishes at both endpoints, and is therefore nonnegative.

As the function $x \mapsto x\log(x-1)$ is convex on $[2,\infty)$,
Jensen's inequality gives
\begin{equation}\label{eq:Jensen}
    \sum_{v\in V}d_v\log(d_v-1)
    \geq
    2m\log\left(\frac{2m}{n}-1\right).
\end{equation}

Now, let
$I=\{v\in V:d^{+}_v>d^{-}_v\}$. Since
$\sum_{v\in V}d^{-}_v=\sum_{v\in V}d^{+}_v=m$, we have $\sum_{v\in V}\min\{d^{-}_v,d^{+}_v\}
    =m-\frac{1}{2}\sum_{v\in V}|d^{-}_v-d^{+}_v|
    =m-\sum_{v\in I}(d^{+}_v-d^{-}_v)$. Note that $\sum_{v\in I}(d^{+}_v-d^{-}_v)$ is the difference between the number of directed edges from $I$ to $V\setminus I$
and the number of directed edges from $V\setminus I$ to $I$. Therefore,
it is at most the number of edges in the cut $(I,V\setminus I)$,
which is at most $b(G)$. Hence $\sum_{v\in V}\min\{d^{-}_v,d^{+}_v\}\geq m-b(G)$.

Combining this inequality with \eqref{eq:AMGM} and \eqref{eq:Jensen}, we obtain 
$$
    \log\frac{s(\og)}{m}
    \geq
    2\log\left(\frac{2m}{n}-1\right)
    +\left(1-\frac{b(G)}{m}\right)
    \log\left(\frac{3}{2}\right). 
$$  \qed
\end{proof}

The following inequality immediately follows from Lemma~\ref{lem:consequences_of_FNV} and Lemma~\ref{lem:oriented_lb}.

\begin{lemma}
\label{lem:main}
If $G$ is a $C_6$-free graph of minimum degree at least two, then 
$$
  m \left(\frac{2m}{n}-1 \right)^{2}\left(\frac{3}{2} \right)^{1-b_n/m} \leq n(n-1) + 80 \Delta m,  
$$
where $b_n = \frac{1}{2} n^{4/3} + 16n$.
\end{lemma}

To apply Lemma~\ref{lem:main}, we must control the maximum degree \(\Delta\). By Lemma~\ref{lem:FNV_collection}(a), this follows once we obtain a sufficiently large minimum degree. We achieve this by passing to an appropriate induced subgraph via the following argument. Fix $a>0$ and $b\geq 0$. For any graph $G=(V(G),E(G))$, define
$$
    \Psi_{a,b}(G)
    =
    |E(G)|-a|V(G)|^{4/3}-b|V(G)|.
$$

\begin{lemma}
\label{lem:induced_subgraph}
Let $G$ be a graph with $\Psi_{a,b}(G)>0$, and let $H$ be an induced
subgraph of $G$ maximizing $\Psi_{a,b}$. Then $H$ is nonempty.
Moreover, every vertex of $H$ has degree at least
$$
    a\left(|V(H)|^{4/3}
    -\bigl(|V(H)|-1\bigr)^{4/3}\right)+b.
$$
Furthermore, if $(G_N)_{N\geq 0}$ is a sequence of graphs such that
$\Psi_{a,b}(G_N)\to\infty$, and $H_N$ is an induced subgraph of $G_N$
maximizing $\Psi_{a,b}$, then $|V(H_N)|\to\infty$.
\end{lemma}

\begin{proof}
Since $\Psi_{a,b}(H) \geq \Psi_{a,b}(G)>0=\Psi_{a,b}(\emptyset)$,
the maximizing subgraph $H$ is nonempty. For any $v\in V(H)$, the
maximality of $\Psi_{a,b}(H)$ gives $\Psi_{a,b}(H)\geq\Psi_{a,b}(H-v)$. Therefore, $$
    |E(H)|-a|V(H)|^{4/3}-b|V(H)|
    \geq
    |E(H-v)|-a\bigl(|V(H)|-1\bigr)^{4/3}
    -b\bigl(|V(H)|-1\bigr).$$
Since $|E(H)|-|E(H-v)|=d_H(v)$,
it follows that $$
    d_H(v)
    \geq
    a\left(|V(H)|^{4/3}
    -\bigl(|V(H)|-1\bigr)^{4/3}\right)+b.$$

For the last assertion, the choice of $H_N$ gives $\Psi_{a,b}(H_N) \geq \Psi_{a,b}(G_N) \longrightarrow\infty$. On the other hand, since $a>0$ and $b\geq 0$,
$$
    \Psi_{a,b}(H_N)
    \leq
    |E(H_N)|
    \leq
    \binom{|V(H_N)|}{2}.
$$
Consequently, $|V(H_N)|\to\infty$.
\qed
\end{proof}

\medskip
\noindent\textit{Proof of Theorem~\ref{thm:main}. } 
The function
$h(t)= 4t^3(3/2)^{1-1/(2t)}$ is strictly increasing, for
$t\geq 1/2$. Moreover, its value at $t=1/2$ is $1/2$, while its
value at $t=2/3$ is greater than $1$. Hence there is a unique
$\alpha\in(1/2,2/3)$ such that
$4\alpha^3(3/2)^{1-1/(2\alpha)}=1$. This is the constant appearing
in Theorem~\ref{thm:main}.

For $n\in\mathbb{N}$, let $b_n=\frac{1}{2}n^{4/3}+16n$ and define
$$
F_n(x)=x(2x/n-1)^2(3/2)^{1-b_n/x}.
$$
Note that $F_n$ is strictly increasing on $(n/2,\infty)$.

Fix a sufficiently large constant $c>0$ satisfying
\begin{equation*}
    c\left(
        \frac{3}{\alpha}
        +\frac{\log(3/2)}{2\alpha^2}
    \right)
    -\frac{1+16\log(3/2)}{\alpha}
    >
    \frac{11520}{\alpha^2}+1.
\end{equation*}

Suppose, for a contradiction, that there are arbitrarily large
integers $N$ for which a $C_6$-free graph $G_N$ on $N$ vertices
satisfies $|E(G_N)|\geq\alpha N^{4/3}+2cN$. Apply
Lemma~\ref{lem:induced_subgraph} with $a=\alpha$ and $b=c$, and let $H_N$ be an
induced subgraph of $G_N$ maximizing $\Psi_{\alpha,c}$. For
simplicity, let $H=H_N$ and $n_h=|V(H)|$. Since
$\Psi_{\alpha,c}(G_N)
=|E(G_N)|-\alpha N^{4/3}-cN
\geq cN\to\infty$,
Lemma~\ref{lem:induced_subgraph} gives $n_h\to\infty$. It also gives
$\delta(H)
    \geq
    \alpha\left(
        n_h^{4/3}-(n_h-1)^{4/3}
    \right)+c$. In particular, $\delta(H)\geq2$ for all sufficiently large $N$.

By the mean value theorem,
$n_h^{4/3}-(n_h-1)^{4/3}
\geq\frac{4}{3}(n_h-1)^{1/3}
\geq n_h^{1/3}$ for $n_h\geq2$. Hence, for all sufficiently large
$N$, we have
$\delta(H)-4
\geq\alpha n_h^{1/3}-4
\geq\frac{2\alpha}{3}n_h^{1/3}$.
It follows from Lemma~\ref{lem:FNV_collection}(a) that
$\Delta(H)\leq(144/\alpha^2)n_h^{1/3}$.

By the upper bound of \eqref{eq:FNVbound},
$|E(H)|<n_h^{4/3}$ for all sufficiently large $N$. Therefore,
$80\Delta(H)|E(H)|
\leq(11520/\alpha^2)n_h^{5/3}$. Furthermore,
$\Psi_{\alpha,c}(H)\geq\Psi_{\alpha,c}(G_N)>0$, and hence
$|E(H)|>\alpha n_h^{4/3}+cn_h$. Put
$x=\alpha n_h^{4/3}+cn_h$. Since $\alpha>1/2$, we have $x>n_h/2$.
Therefore, Lemma~\ref{lem:main} and the monotonicity of $F_{n_h}$
give
\begin{align*}
    F_{n_h}(x)
    &\leq
    F_{n_h}\bigl(|E(H)|\bigr)\\
    &\leq
    n_h(n_h-1)+80\Delta(H)|E(H)|\\
    &\leq
    n_h(n_h-1)
    +\frac{11520}{\alpha^2}n_h^{5/3}.
\end{align*}

On the other hand, substituting
$x=\alpha n_h^{4/3}+cn_h$ into the definition of $F_{n_h}$ gives
\begin{align*}
    \frac{F_{n_h}(x)}{n_h^2}
    &=\Upsilon(n_h^{-1/3}),
\end{align*}
where
\begin{align*}
    \Upsilon (s)=\left(\alpha+cs\right)
    \left(
        2\left(\alpha+cs\right)-s
    \right)^2
        \exp\left[
        \log\left(\frac{3}{2}\right)
        \left(
            1-
            \frac{1/2+16s}
                 {\alpha+cs}
        \right)
    \right]
\end{align*}
Note that $\Upsilon(0)=4\alpha^3(3/2)^{1-1/(2\alpha)}=1$. Differentiating $\Upsilon(s)$ at $s=0$ gives $$\Upsilon^{'}(0)=c\left(
            \frac{3}{\alpha}
            +\frac{\log(3/2)}{2\alpha^2}
        \right)
        -\frac{1+16\log(3/2)}{\alpha}
    $$ Hence by Taylor's expansion we get, 
\begin{align*}
    \frac{F_{n_h}(x)}{n_h^2}=\Upsilon(n_h^{-1/3})
    &=1+
    \left[
        c\left(
            \frac{3}{\alpha}
            +\frac{\log(3/2)}{2\alpha^2}
        \right)
        -\frac{1+16\log(3/2)}{\alpha}
    \right]n_h^{-1/3} +O\left(n_h^{-2/3}\right)
\end{align*} 

By the choice of $c$, for all sufficiently large $N$ this implies
$F_{n_h}(x)>
n_h^2+(11520/\alpha^2+1/2)n_h^{5/3}$, a contradiction. Therefore, $\ex(n,C_6) \leq \alpha n^{4/3}+2cn$ for all sufficiently large $n$, which completes the proof of Theorem~\ref{thm:main}.
\qed

\section{Conclusions}

We improve the upper bound for the Tur\'an number of the hexagon to
$$
\ex(n,C_6)\leq (0.614314\ldots)\, n^{4/3}+O(n).
$$
Together with the construction of F\"uredi, Naor, and
Verstra\"ete~\cite{FNV2006}, this leaves a gap between the best known
lower and upper bounds for the leading constant of $\ex(n,C_6)$.

A natural direction for further work is to determine whether
$\ex(n,C_6)/n^{4/3}$ has a limit and, if so, to determine its value.
It would also be interesting to study whether the path-counting
and orientation arguments used here can be strengthened further,
or adapted to obtain improved upper bounds for $C_{10}$ and other even
cycles.

\bibliographystyle{splncs04}
\bibliography{ref}

\end{document}